\documentclass[12pt]{amsart} 
\usepackage[margin=1.2in]{geometry}

\usepackage{graphicx}
\usepackage{multirow}
\usepackage[table,xcdraw]{xcolor}
\usepackage{graphicx,natbib}
\usepackage{amsmath}
\usepackage{placeins}
\usepackage{amssymb}
\usepackage{epstopdf}
\usepackage{pdfpages}
\usepackage{amsthm}
\usepackage{xcolor}
\usepackage{setspace}
\usepackage{amsaddr}
\usepackage{hyperref}

\newtheorem{theorem}{Theorem}[section]

\newtheorem{proposition}[theorem]{Proposition}
\newtheorem{corollary}[theorem]{Corollary}
\newtheorem{example}[]{Example}[section]
\newtheorem{definition}[theorem]{Definition}
\newtheorem{remark}[theorem]{Remark}

\newcommand*\Prob{\mathop{}\!\mathbb{P}}
\newcommand*\E{\mathop{}\!\mathbb{E}}

\usepackage{multicol}
\usepackage{bm}
\newcommand{\vect}[1]{\pmb{#1}}
\newcommand{\mat}[1]{\boldsymbol{\bm #1}}

\title{Fractional Inhomogeneous Multi-state Models in Life Insurance}
\date{\today}

\author[M. \smash{Bladt}]{Martin Bladt}
\address[M. Bladt]{Department of Actuarial Science, Faculty of Business and Economics, University of Lausanne, CH-1015 Lausanne, Switzerland}

\email{{martin.bladt@unil.ch}}

\begin{document}
\maketitle
\begin{abstract}
In this paper, we demonstrate through the use of matrix calculus a transparent analysis of fractional inhomogeneous Markov models for life insurance where transition matrices commute. The resulting formulae are intuitive matrix generalizations of their single-state counterparts, and the absorption times are matrix versions of well-known scalar distributions. A further advantage of this approach is that it allows extending the analysis to the non-Markovian case where sojourns are Mittag-Leffler distributed, and where the absorption times are fractional phase-type distributed. Considering deterministic time transforms gives rise to fractional inhomogeneous phase-type distributions as absorption times. The latter underlying processes are an example of a regime where not only the present but also the history of a policyholder influences its future evolution. The sub-exponential nature of stable distributions translates into the multi-state insurance model as a random longevity risk at any given state of the chain.
\end{abstract} 

\section{Introduction}

The use of Markov chains in life insurance has been around for many decades (cf. \cite{hoem1969markov}), and their use has had a continued interest in multi-state life insurance models. Of main interest is the calculation of reserves, for which \cite{norberg1991reserves} is a classical reference (see also \cite{norberg1995differential} for higher moments, and \cite{hesselager1996probability} for the distribution of discounted future payments). The now-standard setup considers inhomogeneous Markov models for future payments, and numerical calculation of the reserve and premium calculation can be carried out effectively solving ordinary differential equations. Recently, \cite{bladt2020matrix} introduced product integrals as a means to obtain moments of discounted future payments in a very general Markov setting. The latter reference presents multi-state life insurance calculations using matrix algebra and calculus only, and this approach will also be taken in the present paper. 

We {\color{black} first} consider in this paper the calculation of the reserve for multi-state models. For the case of inhomogeneous Markov models, we consider parametric families which give rise to tractable absorption distributions in terms of inhomogeneous phase-type distributions (cf. \cite{albrecher2019inhomogeneous}). The setting can in principle be deduced from the general theory in \cite{bladt2020matrix} (corresponding to transition matrices that commute). Still, the present formulae are novel, and the principle of duality as introduced below is unique to this setting. 

Subsequently, we refine our analysis and consider a time-fractional version of the previous setup. {\color{black} Time-fractional models and their related Mittag-Leffler distributed inter-event times are a natural generalization of Markov models with exponential sojourn times, and one of the only known cases where the absorption time is fully explicit, due to the calculations being so similar to the exponential case but using fractional calculus.} Multi-state models for time-fractional jump processes, or inverse subordinated Markov processes, have been considered in \cite{donatien}, where the jump process of \cite{mml} is viewed from a life insurance perspective. The absorption times are fractional phase-type distributed (cf. \cite{mml}, and \cite{mmml, mfph}), and have been successfully adapted to model non-life insurance data. However, the latter model is, in a sense, not immediately suitable to model mortality data since they do not possess any moment of order $k>\alpha$, where $0<\alpha<1$ in the fractional case (when $\alpha=1$ in the Markov case, there is a regime-switch, and there are moments of all orders). In particular, they do not possess a mean, due to their very heavy tails.

{\color{black} We propose} inhomogeneous versions of the fractional process of \cite{mml}, that is, we deterministically time-transform the chain with a function $g$, {\color{black}which gives rise to the main novelties of the contribution}. We introduce a new class of distributions for their absorption times, namely the fractional inhomogeneous phase-type distributions ($\mbox{IPH}_\alpha$, $\alpha\in(0,1]$) which for parametric transforms lead to finite-moments distributions, {\color{black}addressing the infinite mean drawback of a time-fractional chain}. For the particular case $\alpha=1$ we reduce to the usual inhomogeneous Markov setup, and for the specific case $g(t)=t$ we recover the form of \cite{donatien}.

When generalizing the results to the time-fractional case, it is possible to partially or fully retrieve analogous formulae to the Markov case. The reserve for time $t=0$ is fully explicit in terms of functions of matrices, which can be computed efficiently using Matrix algebra for very general structures for the underlying intensity matrix (for instance, for Generalized Coxian distributions). For $t>0$, only the conditional reserve is available, and the actual reserve would have to be simulated.

There are several features that the time-fractional setup possesses which are not present in the Markov case. Firstly, the chain is no longer memoryless, and thus the evolution of an individual in a given state does not only depend on the present but also the past. Secondly, the sub-exponential behaviour (before the deterministic time-transform with $g$) of the sojourn times concentrates the distribution of the absorption time more on both tails. It hence serves as a mathematical model for individuals who traverse a state unusually quickly or unusually slowly. {\color{black} In essence, the last condition can be interpreted as a non-traditional longevity risk from the insurers' perspective}. {{\color{black} Finally, fractional specifications and their generalizations are examples of semi-Markov models which share the same mathematical toolbox for deriving explicit formulae with their Markov counterparts. This cannot be said about other semi-Markov constructions. This property suggests the use of fractional distributions not only as mathematical objects, but also as well-motivated and physically interpretable models.}
}

{\color{black} The use of inhomogeneity transforms for Markov models and their time-fractional counterparts has recently been studied in several contexts relevant to insurance applications, mainly related to non-life insurance. To name a few, \cite{albrecher2019inhomogeneous} deal with inhomogeneous Markov models, \cite{bladt2020inh} consider an extension which incorporates covariates, and \cite{mml} treat a power-transformed fractional chain. One of the common takeaways is that the use of an inhomogeneity function can help with the correct tail and moments specifications, which in turn makes a statistical analysis feasible with just a few underlying states. The present contribution brings these models to life insurance, where the underlying sub-exponential behaviour can be interpreted as a longevity risk, but where the regularly-varying tails are tempered by the inhomogeneity function. In the new scale, the transformed sub-exponential behaviour is not only apparent in the right tail, but also in the left tail, resulting in an ``impermanence" risk, which is relevant to lump sums paid at transitions to the absorbing state. }

The remainder of the paper is structured as follows. In Section 2 we present some basic properties of homogeneous Markov models and phase-type distributions. In Section 3, we briefly survey inhomogeneous Markov jump processes and inhomogeneous phase-type distributions, before turning into the derivation of the reserve formulae. We also provide an interpretation of the terms involved in the formula for the reserve in terms of dual distributions and consider some examples for parametric $g$ functions. Section 4 is devoted to the study of the time-fractional generalization of the inhomogeneous Markov models. We provide a short reminder of the basic properties of the underlying fractional chain, {\color{black} a motivational study, and then consider the deterministically-transformed versions which naturally lead to the definition of fractional inhomogeneous phase-type distributions and the derivation of the reserve. In section $5$ we provide a numerical example of life insurance contracts with underlying fractional structures. We conclude in Section 6. }

\section{{\color{black}Preliminaries}}

Consider a finite state-space $E=\{1,\dots, p\}$ corresponding to the different states of a person in a life-insurance contract. These states are for convenience labeled as natural numbers, but can represent for instance: alive, disability, illness, death. Consider also a Markov jump process $(Z_t)$ evolving in $E$, with transition probabilities
\begin{align*}
p_{i,j}(s,t)=\Prob(Z_t=j|Z_s=i),
\end{align*}
which in matrix notation amount to the transition matrix $\mat{P}(s,t)=(p_{ij}(s,t))_{i,j=1,\dots p}$. This process models the dynamics of an insured policyholder within the different states from a classical point of view, and in particular depends on the transition rates and the sojourn time distributions. The Markov property implies that the Chapman-Kolmogorov equations are satisfied, which in turn gives an explicit form of $\mat{P}$, as the following matrix exponential 
\begin{align}\label{Markov_transition_mat}
\mat{P}(s,t)=\exp(\mat{\Lambda}(t-s)),
\end{align}
where $\mat{\Lambda}$ is an intensity matrix. The latter satisfies that $\lambda_{ij}>0$ for $i\neq j$, which correspond to the (un-nomalized) jump rates between states, and $-\lambda_{ii}=\sum_{i\neq j}\lambda_{ij}$ is the rate of the exponential sojourn time at state $i\in E$.

Throughout the rest of the paper, if $u$ is an analytic function and $\mat{A}$ is a matrix, we define
\begin{align*}
	u( \mat{A})=\dfrac{1}{2 \pi i} \oint_{\gamma}u(w) (w \mat{I} -\mat{A} )^{-1}dw \,,
\end{align*}
where $\gamma$ is a simple path enclosing the eigenvalues of $\mat{A}$ and $\mat{I}$ is the identity matrix of the same dimension; cf.\ \cite[Sec. 3.4]{Bladt2017} for details. In particular, the matrix exponential can be defined in this way, and is equivalent to the more standard definition in terms of its series representation
\begin{align*}
\exp(\mat{A})=\sum_{i=1}^{n}\frac{\mat{A}^n}{n!}.
\end{align*}

Equation \eqref{Markov_transition_mat} allows for a very transparent treatment of the process $(Z_t)$ and related functionals. We now provide an important example. If the state space consists of $p+1$ points, where the last state is an absorbing one, and all other states are transient, then the intensity matrix simplifies in the following manner

\begin{align}\label{ph_structure}
	\mat{\Lambda}= \left( \begin{array}{cc}
		\mat{T} &  \vect{t} \\
		\vect{0} & 0
	\end{array} \right)\,.
\end{align}
where $\mat{T} $ is a $p \times p$ sub-intensity matrix and $\vect{t}$ is a $p$--dimensional column vector satisfying $\vect{t} =- \mat{T} \, \vect{e}$, where $\vect{e} $ is the $p$--dimensional column vector of ones. Let $ \pi_{k} = \Prob(Z_0 = k)$, $k = 1,\dots, p$, $\vect{\pi} = (\pi_1 ,\dots,\pi_p )$ be a probability vector corresponding to the initial distribution of the process $(Z_t)$. Notice that $\Prob(Z_0 = p + 1) = 0$. Then we say that the time until absorption 
\begin{align*}
	\tau_Z = \inf \{ t \geq  0 \mid Z_t = p+1 \}
\end{align*}
is a phase--type distribution with parameters $\vect{\pi}$ and $\mat{T}$, and simply write

$$\tau \sim  \mbox{PH}(\vect{\pi} , \mat{T} ).$$

The asymptotic behaviour in the tail of a phase--type distribution depends on the eigenvalues of $\mat{T}$, and is precisely given by
\begin{align*}
\Prob(\tau_Z>y)=\sum_{j=1}^{m} \sum_{l=0}^{\kappa_{j}-1}y^{l} \mathrm{e}^{\Re\left(-\eta_{j}\right) y }\left[a_{j l} \sin \left(\Im\left(\eta_{j}\right) y \right)+b_{j l} \cos \left(\Im\left(\eta_{j}\right) y \right)\right], 
\end{align*}
where $\eta_j$ are the eigenvalues of the Jordan blocks $\mat{J}_{j}$ of $\mat{T}$, with corresponding dimensions $\kappa_j$, $j = 1,\dots, m$, and $a_{jl}$ and $b_{jl}$ are constants depending on $\vect{\pi}$ and $\mat{T}$. If $\eta$ is the largest real eigenvalue of $\mat{T}$ and $n$ is the dimension of the Jordan block of $\eta$ it follows that
\begin{align}\label{PHtail_asymptotic}
	\Prob(\tau_Z>y) \sim c y^{n -1} e^{\eta y} \,, \quad y \to \infty \,,
\end{align}
where $c$ is a positive constant. In other words, every phase--type distribution has an exponential tail, with second order behaviour akin to the Erlang distribution. A comprehensive and modern treatment of phase--type distributions is found Bladt \& Nielsen \cite{Bladt2017}.

Thus, any life-insurance model based on this classical construction will necessarily have asymptotically exponential mortality rates, which in real-life datasets if often not fulfilled. Concerning the body of the distribution, phase--type distributions are dense in the sense of weak convergence on the positive real line, so for lower quantiles they are well-suited for modelling any possible time-of-death distribution. 

Notice, however, that the denseness property requires arbitrary growth of the dimension $p$, which conflicts with the interpretational aspect of the process $(Z_t)$ as evolving through different policy-holder statuses. A possible resolution to this inconvenience is to consider $p$ \textit{blocks} of states, each of arbitrary dimension. The interpretation being that within each block the policyholder has one status, but the time spent with that status evolves through \textit{hidden states}.

\section{The Inhomogeneous Markov Model}

In this section we consider a process $(X_t)$ in the state space $$E^\ast=(\vect{e}_1,\dots, \vect{e}_{p}, \dagger)$$ where the transient states $\vect{e}_k$ are the $k$-th unit basis vectors in $\mathbb{R}^p$, and $\dagger$ is an absorbing state. The process $(X_t)$ is defined in terms of $(Z_t)$ and an inhomogeneity function, i.e. an increasing function $g:\mathbb{R}_+\to \mathbb{R}_+$ with $g(0)=0$, as follows
\begin{align*}
X_t=(1\{Z_{g^{-1}(t)}=1\},\dots, 1\{Z_{g^{-1}(t)} = p\}), \quad \mbox{if}\quad Z_{g^{-1}(t)}\neq p+1,\quad t\ge0,
\end{align*}
and $X_t=\dagger$ otherwise. Loosely speaking, $X_t$ evolves through $E^\ast$ in the same manner as $Z_{g^{-1}(t)}$ evolves through $E$. Observe that $X_0\sim \vect{\pi}$ as well. In other words, the process $(X_t)$ is a deterministicly time-changed version of $(Z_t)$, with time being slowed down or sped up according to $g^{-1}$. For policyholders, this means that the times that they spend in each state is no longer exponential, and depends on the specific form of $g$. This gives rise to more parsimonious models where the block augmentation procedure outlined at the end of the previous section can be greatly reduced and in some cases entirely eliminated. 

The absorption time of $X_t$ can then be expressed in terms of $\tau_Z$ as follows
\begin{align*}
	\tau = \inf \{ t \geq  0 \mid X_{t} = \dagger \}=\inf \{ g(t) \geq  0 \mid Z_{t} = p+1 \}=g(\tau_Z).
\end{align*}
The distribution of these absorption times were studied from a statistical perspective in \cite{Bladt2017} and further extended to parametric and multivariate settings in \cite{albrecher2020iphfit} and recently in a survival analysis setup in \cite{bladt2020inh}, where they have been referred to as inhomogeneous phase--type distributions and denoted by the following parametrization $$\tau\sim \mbox{IPH}(\vect{\pi},\mat{T},\lambda),$$ where
\begin{align*}
\lambda(t)=\frac{d}{dt}g^{-1}(t)
\end{align*}
corresponds to the instantaneous inhomogeneous intensity of $X_t$.

 In \cite{bladt2020inh}, covariate-specific $\tau_x$ has been successfully implemented for a variety of models, the estimation procedure requiring a generalized EM algorithm. In the context of life-insurance, the most natural covariate $x$ is age of the policyholder at inception of the insurance contract.

One important feature of the absorption times of inhomogeneous Markov processes is that their tails are no longer exponential. Indeed, the shape depends on the function $g$, which follows directly from the following result (the proof can be found in \cite{bladt2020inh})
\begin{proposition}
	Let $\tau$ be the absorption time of $X_t$. Then the hazard function $h$ and cumulative hazard function $H$ of $\tau$ satisfy, respectively,  
	\begin{align*}
		h(t) &\sim A \lambda(t) \,, \quad t \to \infty \,,\\
		H(t)& \sim B g^{-1}(t) \,, \quad t \to \infty \,,
	\end{align*}
	where $A,B$ are positive constants.
\end{proposition}
Thus, more truthful asymptotic behaviour can be specified for the mortality rate of (possibly covariate-dependent) policyholders with the appropriate choice of parametric family for $g$.

\subsection{{\color{black}The reserve at time $t$}}
The reserve or provision is a fundamental quantity of interest related life-insurance contracts. Let $(\mathcal{F}_t)_{t\ge0}$ denote the natural filtration of the process $(X_t)$, and introduce the counting process 
\begin{align*}
N^{ij}_t=\#\{s \in(0, t], X(s-) =\vect{e}_i, X(s)=\vect{e}_j\},
\end{align*}
which by definition of $\mat{\Lambda}$ satisfies
\begin{align*}
\E(dN^{ij}_t|\mathcal{F}_{t-})=1\{X_{t-}=\vect{e}_i\}\lambda(t)\lambda_{ij}dt.
\end{align*}
If we assume that premiums are collected at a rate specified by the column vector $\vect{c}=(c_1,\dots,c_p)$, annuities are paid according to the rate column vector  $\vect{a}=(a_1,\dots,a_p)$, and benefits are {\color{black}paid} in lump sums at transition times according to a matrix $\mat{B}=(b_{ij})_{i,j=1,\dots,p}$, with $b_{ii}=0,\:\forall i$, then the reserve {\color{black}at time t with an expiry time n }is given by
\begin{align*}
Y_t=\E\left[ \int_t^n e^{-\int_t^sr}dB_s\Big |\mathcal{F}_t\right],
\end{align*}
where $r$ denotes the interest rate or a related quantity ($\int_a^b r=\int_a^b r(s)ds$), $n$ is the maximal duration of the contract, and
$$d B_t=\sum_{i=1}^p 1\{X_{t}=\vect{e}_i\} (a_i-c_i)dt+\sum_{i\neq j} b_{ij} d N^{ij}_t$$
is the process of total benefits less premiums. An obvious adaptation of the present setup can be made for when a stream of lump sums is, conditionally on the state, independently paid out. For a Poisson process, the term corresponding to its expected value can be absorbed into the vector $\vect{a}$, and thus we focus presently on the simplified case. In the remainder of the paper we will use the shorthand notation $\E_t(\cdot)=\E(\cdot|\mathcal{F}_t)$. The usual assumption is that $Y_0=0$, which in vector notation amounts to 
\begin{align*}
\E_0\left[ \int_0^ne^{-\int_0^ur}X_u (\vect{c}-\vect{a})du \right]=\sum_{i\neq j}\E_0\left[ \int_0^ne^{-\int_0^u r}1\{X_u = \vect{e}_i\}b_{ij}dN^{ij}_u \right].
\end{align*}
At the time of signing, the above equation guarantees that on average the premiums will cover the losses associated to the contract. However, at time $t>0$, $Y_t$ may be different from zero, and this is seen as a liability from the insurer's perspective. In vector notation, we obtain

\begin{align*}
Y_t=\E_t\left[ \int_t^ne^{-\int_t^u r)}X_u (\vect{a}-\vect{c})du \right]+\sum_{i\neq j}\E_t\left[ \int_t^ne^{-\int_t^u r}1\{X_u = \vect{e}_i\}b_{ij}dN^{ij}_u \right].
\end{align*}

With the use of functional calculus, we are able to obtain an alternate form of the reserve for a general inhomogeneity $g$ function. For two matrices $\mat{M},\mat{N}$ of the same dimension, we define $\mat{O}=\mat{M}\cdot\mat{N}$ by $(o_{ij})=(m_{ij}n_{ij})$, i.e. the result of the component-wise product of the two matrices.

\begin{proposition}\label{main_prop}
The reserve at time $t$ {\color{black} with expiry date $n$ }is given by
\begin{align*}
Y_t=X_t \left\{f_1(\mat{\Lambda})(\vect{a}-\vect{c})+f_2(\mat{\Lambda})(\mat{\Lambda}\cdot\mat{B})\vect{e}\right\}
\end{align*}
where
\begin{align}
f_1(w)&=\int_t^ne^{-\int_t^u r+ w(g^{-1}(u)-g^{-1}(t))}du,\label{f11}\\
f_2(w)&=\int_t^ne^{-\int_t^u r+ w(g^{-1}(u)-g^{-1}(t))}\lambda(u)du.\label{f22}
\end{align}
\end{proposition}
\begin{proof}
Observe that the term
\begin{align*}
\E_t\left[ \int_t^ne^{-\int_t^u r}X_u du \right]&= \int_t^ne^{-\int_t^u r}\E_t\left[X_u\right]du\\
&=\int_t^ne^{-\int_t^u r}X_t\mat{P}(t,u)du.
\end{align*}
But it is straightforward to see that for an appropriate path $\gamma$, {\color{black} from \eqref{Markov_transition_mat},}
\begin{align*}
\mat{P}(t,u)&=\exp(\mat{\Lambda}(g^{-1}(u)-g^{-1}(t)))\\
&=\dfrac{1}{2 \pi i} \oint_{\gamma}\exp(w) (w \mat{I} -\mat{\Lambda}(g^{-1}(u)-g^{-1}(t)) )^{-1}dw\\
&=\dfrac{1}{2 \pi i} \oint_{\gamma}\exp(w(g^{-1}(u)-g^{-1}(t))) (w \mat{I} -\mat{\Lambda} )^{-1}dw,
\end{align*}
from which
\begin{align*}
\E_t&\left[ \int_t^ne^{-\int_t^u r}X_u du \right]\\
&=\int_t^ne^{-\int_t^u r}X_t\dfrac{1}{2 \pi i} \oint_{\gamma}\exp(w(g^{-1}(u)-g^{-1}(t))) (w \mat{I} -\mat{\Lambda} )^{-1}dwdu\\
&=\dfrac{1}{2 \pi i} \oint_{\gamma}\int_t^ne^{-\int_t^u r}X_t\exp(w(g^{-1}(u)-g^{-1}(t))) du(w \mat{I} -\mat{\Lambda} )^{-1}dw\\
&=\dfrac{1}{2 \pi i} \oint_{\gamma}X_tf_1(w) du(w \mat{I} -\mat{\Lambda} )^{-1}dw,
\end{align*}
and the first part of the expression follows. Now, conditionally on the event $X_t=\vect{e}_k$, we have that
$$
\E_t\left[ \int_t^ne^{-r(u-t)}1\{X_u = \vect{e}_i\}b_{ij}dN^{ij}_u \right]=\int_t^n e^{-r(u-t)}\,p_{ki}(t,u) b_{ij}\lambda(u)\lambda_{ij}du,$$
from which
\begin{align*}
&\sum_{i\neq j}\E_t\left[ \int_t^ne^{-\int_t^u r}1\{X_u = \vect{e}_i\}b_{ij}dN^{ij}_u \right]\\
&=\int_t^n e^{-\int_t^u r}X_t\exp(\mat{\Lambda}(g^{-1}(u)-g^{-1}(t))) \lambda(u)(\mat{\Lambda}\cdot\mat{B})\vect{e}du\\
&=\dfrac{1}{2 \pi i} \oint_{\gamma}\int_t^ne^{-\int_t^u r}X_t\exp(w(g^{-1}(u)-g^{-1}(t))) \lambda(u)du(w \mat{I} -\mat{\Lambda} )^{-1}dw \,(\mat{\Lambda}\cdot\mat{B})\vect{e}\\
&=\dfrac{1}{2 \pi i} \oint_{\gamma}X_t f_2(w)du(w \mat{I} -\mat{\Lambda} )^{-1}dw \,(\mat{\Lambda}\cdot\mat{B})\vect{e},
\end{align*}
which yields the second term of the expression, and consequently completes the proof.

\end{proof}
The above expression for the reserve dissects the interplay between the inhomogeneity function $g$ and the annuities, benefits and premiums in an intuitive manner. The continuous nature of annuity payments and premium collection allows them to be directly comparable, even in the inhomogeneous case. However, this is not the case for benefits paid out when transitioning between states. 

For instance, if the intensity function $\lambda(t)$ has a spike between $t$ and $n$ which would otherwise not have been there in a homogeneous model, the term $$f_2(\mat{\Lambda})(\mat{\Lambda}\cdot \mat{B})\vect{e}$$ increases. This reflects that there will be an unusual amount of benefits being paid out due to a cluster of jumps. Hence, to maintain $Y_0=0$, $\vect{c}$ must grow in at least one coordinate. More generally, in an inhomogeneous Markov model, the lengthening or shortening of sojourn times across time has to be accounted for in the premium if and only if lump sum benefits are paid at state transitions. Also note that the premium calculation will in general  depend on $r$.

In the special case of $g(x)=x$, the functions $f_1$ and $f_2$ are both equal to the same exponential function, and we recover the homogeneous Markov model. A possible solution to the requirement $Y_0=0$ is then given by the fair premium
\begin{align*}
\vect{c}=\vect{a}+(\mat{\Lambda}\cdot\mat{B})\vect{e},
\end{align*}
which is independent of $f_1=f_2$, and in particular of $r$. 
However, contracts usually reserve the right to collect premiums only on specific states. For instance, when a policyholder is in a state corresponding to unemployment, illness or disability, premiums will seldom be collected. Under such restrictions, the premium collection will in general also depend on $r$.

\subsection{{\color{black}The reserve at time $0$}}
For convenience, we will assume that $E=\{1,\dots, p+1\}$ with the first $p$ states being transient and $p+1$ being an absorbing state and that $r$ is a constant. We specify the structure of the benefit matrix by
\begin{align*}
	\left( \begin{array}{cc}
		\mat{B} &  \vect{b} \\
		\vect{0} & 0
	\end{array} \right)\,,
\end{align*}
with the vector $\vect{b}$ denoting the benefits paid out at death from each one of the distinct states. 
Let us examine the functions $f_1$ and $f_2$ {\color{black} from \eqref{f11} and \eqref{f22}} more carefully at $t=0$, i.e. at the time of inception of the insurance contract. We may rewrite them by a change of variable in the following manner
\begin{align*}
f_1(-w)&=r^{-1}\int_{0}^{g^{-1}(n)}e^{-w u}\left\{re^{-rg(u)} \frac{dg}{du}(u)\right\}du\\
f_2(-w)&=w^{-1}\int_{0}^{n}e^{-r u}\left\{we^{-wg^{-1}(u)} \frac{dg^{-1}}{du}(u)\right\}du,
\end{align*}
from which a striking similarity is apparent, which we will refer to as \textit{duality}. Since both $g$ and $g^{-1}$ are increasing functions starting at zero, the functions $$e^{-rg(u)}, \quad e^{-wg^{-1}(u)}\quad u\ge0,$$
are survival functions, and hence 
\begin{align*}
h_1^{(r)}(u)&=re^{-rg(u)} \frac{dg}{du}(u),\\
h_2^{(w)}(u)&=we^{-wg^{-1}(u)} \frac{dg^{-1}}{du}(u)
\end{align*}
are probability density functions. Consequently, $f_1(-w)$ and $f_2(-w)$ for $w>0$ are both incomplete Laplace integrals. Taking  $n\to \infty$, i.e. allowing for indefinite payments until (possible) absorption of $(X_t)$, typically corresponding to the decease of the policyholder, we obtain the following formulas
\begin{align}\label{laplaces}
f_1(-w)&=r^{-1}\mathcal{L}_{h_1^{(r)}}(w),\nonumber\\
f_2(-w)&=w^{-1}\mathcal{L}_{h_2^{(w)}}(r),
\end{align}
where $\mathcal{L}_h$ denotes the Laplace transform of the function $h$.

As an immediate consequence, we have the following expression for the reserve at time zero of a contract with no expiry date.
\begin{corollary}\label{cor3}
The reserve at time $t=0$ for a contract with no expiry date is given by
\begin{align*}
Y_0=X_0 \left\{r^{-1}\mathcal{L}_{h_1^{(r)}}(-\mat{T})(\vect{a}-\vect{c})+(-\mat{T})^{-1}\mathcal{L}_{h_2^{(-\mat{T})}}(r)[(\mat{T}\cdot\mat{B})\vect{e}+\vect{t}\cdot\vect{b}]\right\}
\end{align*}
\end{corollary}
\begin{proof}
Note that no more premiums or annuities are paid when the absorbing state is reached. Consequently, the expression follows from Proposition \ref{main_prop}, the identities \eqref{laplaces}, and the fact that for $\mat{\Lambda}$ with structure as in \eqref{ph_structure} the matrix exponentials can be decomposed as:
\begin{align*}
	\exp(\mat{\Lambda}s)= \left( \begin{array}{cc}
		\exp(\mat{T}s) &  \vect{e}-\exp(\mat{T}s)\vect{e} \\
		\vect{0} & 1
	\end{array} \right)\,
\end{align*}
and by block matrix multiplication
\begin{align*}
	\exp(\mat{\Lambda}s)\mat{\Lambda}= \left( \begin{array}{cc}
		\exp(\mat{T}s)\mat{T} &  \exp(\mat{T}s)\vect{t} \\
		\vect{0} & 0
	\end{array} \right)\,.
\end{align*}
\end{proof}

\begin{remark}\normalfont
In the above corollary, $\mathcal{L}_{h_1^{(r)}}(-\mat{T})$ is well defined since the real part of the eigenvalues of $\mat{T}$ are negative. Moreover, $-\mat{T}$ is known to be invertible and $-\mat{T}^{-1}$ is known as the Green matrix.
\end{remark}

We now give several illustrations of the above result.
\begin{example}\normalfont
(Power transform/Matrix-Weibull) Consider the parametric function $g_\theta(x)=x^{1/\theta}$. Then
\begin{align*}
h_1^{(r)}(u)=\frac{r}{\theta}e^{-ru^{1/\theta}}u^{1/\theta-1}
\end{align*}
is a {\color{black}Weibull$(r, 1/\theta)$} density, and so
\begin{align*}
\mathcal{L}_{h_1^{(r)}}(-\mat{T})=\int_{0}^{\infty} e^{\mat{T}u} \frac{r}{\theta}e^{-ru^{1/\theta}}u^{1/\theta-1}d u.
\end{align*}
Similarly, $h_2^{(w)}(u)$ is a Weibull density with $(r,\, 1/\theta)$ replaced by $(w,\theta)$. Hence,
\begin{align*}
\mathcal{L}_{h_2^{(-\mat{T})}}(r)=-\int_{0}^{\infty} e^{-r u} \theta\mat{T}e^{\mat{T}u^{\theta}}u^{\theta-1}d u.
\end{align*}
\end{example}

\begin{example}\normalfont
(Matrix-Pareto) Consider the function $g_\beta(x)=\beta(e^{x}-1)$. Then
\begin{align*}
h_1^{(r)}(u)=r\beta e^{-r\beta(e^{u}-1)+ u}
\end{align*}
is a Gompertz density, and so some integration gives
\begin{align*}
\mathcal{L}_{h_1^{(r)}}(-\mat{T})=r\beta e^{r \beta}\mathtt{E}_{-\mat{T}}(r\beta),
\end{align*}
where
\begin{align*}
\mathtt{E}_a(t)=\int_1^\infty e^{-tv}v^{-a}dv.
\end{align*}
The dual distribution is given by
\begin{align*}
h_2^{(w)}(u)=we^{-w\log(\beta^{-1} u+1)}\frac{\beta^{-1}}{ \beta^{-1}u+1}=w\frac{\beta^{w}}{( u+\beta)^{w +1}} .
\end{align*}
and we recognize the shifted Pareto distribution. Further integration yields
\begin{align*}
\mathcal{L}_{h_2^{(-\mat{T})}}(r)=-\mat{T}e^{\beta r}{\beta^{-\mat{T}}}\mathtt{E}_{-\mat{T}+1}(r).
\end{align*}
Collecting terms gives
\begin{align*}
Y_0=X_0e^{r \beta}\left\{\beta \mathtt{E}_{-\mat{T}}(r\beta)(\vect{a}-\vect{c})+{\beta^{-\mat{T}}}\mathtt{E}_{-\mat{T}+1}(r)[(\mat{T}\cdot\mat{B})\vect{e}+\vect{t}\cdot\vect{b}]\right\}.
\end{align*}
The terms multiplying $(\vect{a}-\vect{c})$ and $(\mat{T}\cdot\mat{B})\vect{e}$ both diverge as $\beta\to\infty$, reflecting increasingly longer sojourn times at each state.
\end{example}

\begin{example}\normalfont \label{matrix-gompertz}
(Matrix-Gompertz)
Consider the function $g_{\color{black}\kappa}(u)={\color{black}\kappa}^{-1}\log({\color{black}\kappa} u +1)$. Then 
\begin{align*}
h_1^{(r)}(u)=r e^{-r{\color{black}\kappa}^{-1}\log({\color{black}\kappa} u +1)}\frac{1}{{\color{black}\kappa} u +1}=\frac{r}{({\color{black}\kappa} u +1)^{r{\color{black}\kappa}^{-1}+1}},
\end{align*}
which is a shifted Pareto density, and then
\begin{align*}
\mathcal{L}_{h_1^{(r)}}(-\mat{T})=re^{-\mat{T}{\color{black}\kappa}^{-1} }{{\color{black}\kappa}^{-r{\color{black}\kappa}^{-1}}}\mathtt{E}_{r{\color{black}\kappa}^{-1}+1}(-\mat{T}).
\end{align*}
The dual distribution is given by
\begin{align*}
h_2^{(w)}(u)=w e^{-w{\color{black}\kappa}^{-1}(e^{{\color{black}\kappa} u}-1)+ {\color{black}\kappa} u},\end{align*}
which is of Gompertz type, and thus
\begin{align*}
\mathcal{L}_{h_2^{(-\mat{T})}}(r)=-\mat{T}{\color{black}\kappa}^{-1} e^{-\mat{T}{\color{black}\kappa}^{-1}}\mathtt{E}_{r{\color{black}\kappa}^{-1}}(-\mat{T}{\color{black}\kappa}^{-1}).
\end{align*}
Hence, we obtain for the reserve
\begin{align*}
Y_0=X_0e^{-\mat{T}{\color{black}\kappa}^{-1}}\left\{ {{\color{black}\kappa}^{-r{\color{black}\kappa}^{-1}}}\mathtt{E}_{r{\color{black}\kappa}^{-1}+1}(-\mat{T})(\vect{a}-\vect{c})+ {\color{black}\kappa}^{-1} \mathtt{E}_{r{\color{black}\kappa}^{-1}}(-\mat{T}{\color{black}\kappa}^{-1})[(\mat{T}\cdot\mat{B})\vect{e}+\vect{t}\cdot\vect{b}]\right\}.
\end{align*}

This example can be considered as the dual of the previous example, since the $h_1^{(r)}$ and $h_2^{(w)}$ functions are interchanged. Thus, the Gompertz and Pareto distributions can be seen as playing antipodal roles in the calculation of the reserve.

The Gompertz distribution, and its matrix generalization, have been empirically demonstrated to be suitable models for mortality modeling. This makes $g_{\color{black}\kappa}$ a strong candidate for the inhomogeneity transform from a statistical point of view in life insurance. 

\end{example}

\begin{example}\normalfont
(Constant annuities less premiums) Consider the reserve at any time $t$, possibly with a fixed expiry at time $n$, and arbitrary $g$. If each element of the vector $\vect{a}-\vect{c}$ is equal to the same constant value $k$ then, {\color{black} from Proposition \ref{main_prop},}
\begin{align*}
X_t f_1(\mat{\Lambda})(\vect{a}-\vect{c})&=k X_tf_1(\mat{T})\vect{e}\\
&=k X_tf_1(\mat{T})\vect{e}\\
&=k\int_t^ne^{-r(u-t)}\left[X_te^{\mat{T}(g^{-1}(u)-g^{-1}(t))}\vect{e}\right]du.
\end{align*}
The term in square brackets is computationally straightforward to evaluate without needing to resort to functional calculus of integrals, and hence $Y_t$ can be efficiently obtained. Moreover, for $t=0$, one may write the above expression as
\begin{align*}
k\int_t^ne^{-r(u-t)}\overline F_X(u)du
\end{align*}
where $\overline F$ is the tail of a random variable $$X\sim\mbox{IPH}(X_t,\mat{T},\lambda).$$
\end{example}

\subsection{Numerical evaluation of matrix functions}
The reserve associated to an insurance contract on an inhomogeneous Markov process with $p$ transient states and one absorbing one is readily seen from the above developments to be given for $t\in[0,\infty)$, and $n\in(t,\infty]$, by
\begin{align*}
Y_t=X_t \left\{f_1(\mat{T})(\vect{a}-\vect{c})+f_2(\mat{T})[(\mat{T}\cdot\mat{B})\vect{e}+\vect{t}\cdot\vect{b}]\right\},
\end{align*}
where $\mat{T}$ is as in \eqref{ph_structure}.

 The evaluation of the terms $f_1(\mat{T})$ and $f_2(\mat{T})$ can be found explicitly for special cases, as was seen in above. However, in full generality, a method for their numerical evaluation is required. This subsection discusses one possible strategy.

 Assume that $\mat{T}$ has a Jordan normal form given by $$\mat{T}=\mat{P}\,\mat{D}\,\mat{P}^{-1}$$ where   $\mat{D}= \mbox{diag}(\mat{J}_1,...,\mat{J}_r)$ and
\[  \mat{J}_i = 
\begin{pmatrix}
  \eta_i & 1 & 0 & \cdots & 0 \\
  0 & \eta_i & 1 & \cdots & 0 \\
  0 & 0 & \eta_i & \cdots & 0 \\
  \vdots & \vdots & \vdots & \vdots\vdots\vdots & \vdots \\
  0 & 0 & 0 & \cdots & \eta_i
\end{pmatrix} ,
  \]
  with the total dimension of the blocks adding up to $p$. Then we may express $f_k(\mat{T})$, $k =1,2$ by
  \begin{align}\label{jorddec_f}
     f_k(\mat{T})=\mat{P}\,\mbox{diag}(f_k(\mat{J}_1),...,f_k(\mat{J}_r)) \,\mat{P}^{-1}, 
     \end{align}
  where
  \begin{align} \label{jorddec_f2} f_k(\mat{J}_i) =  
\begin{pmatrix}
f_k(\eta_i) & f^{(1)}_k(\eta_i) & \frac{f_k^{(2)}(\eta_i)}{2!} & \cdots & \frac{f_k^{(\kappa_i-1)}(\eta_i)}{(\kappa_i-1)!} \\
0 & f_k(\eta_i) & f_k^{(1)}(\eta_i) & \cdots & \frac{f_k^{(\kappa_i-2)}(\eta_i)}{(\kappa_i-2)!} \\
0 & 0 & f_k(\eta_i) & \cdots & \frac{f_k^{(\kappa_i-3)}(\eta_i)}{(\kappa_i-3)!} \\
\vdots & \vdots & \cdots & \vdots\vdots\vdots & \vdots \\
0 & 0 & 0 & \cdots & f_k(\eta_i)
\end{pmatrix},
   \end{align}
 with {\color{black}
 \begin{align*}
 f_1^{(m)}(w)=\frac{d^mf_1}{dw^m}(w)=\int_t^n\{g^{-1}(u)-g^{-1}(t)\}^me^{-r(u-t)+w (g^{-1}(u)-g^{-1}(t))}du\\
  f_2^{(m)}(w)=\frac{d^mf_2}{dw^m}(w)=\int_t^n\{g^{-1}(u)-g^{-1}(t)\}^me^{-r(u-t)+w (g^{-1}(u)-g^{-1}(t))}\lambda(u)du,
 \end{align*}}
and  $\kappa_i$ the dimension of $\mat{J}_i$.  Consequently, we require only to evaluate $f_k$ and its derivatives at the different eigenvalues of $\mat{T}$.

A special and important case is when all the eigenvalues are real and distinct, which can be verified for phase--type distributions such as Erlang, Coxian, or Generalized Coxian, the latter being dense on the set of all distributions on the positive real line. In that case, \eqref{jorddec_f} and \eqref{jorddec_f2} reduce to

\begin{align*}
f_1&(\mat{T})=\mat{P}\, 
  \mbox{diag}\left( \int_t^ne^{-r(u-t)+\eta_i (g^{-1}(u)-g^{-1}(t))}du, \: i=1,\dots,p \right) \mat{P}^{-1},\\
  f_2&(\mat{T})=\mat{P}\, 
  \mbox{diag}\left( \int_t^ne^{-r(u-t)+\eta_i (g^{-1}(u)-g^{-1}(t))}\lambda(u)du, \: i=1,\dots,p \right) \mat{P}^{-1}.
\end{align*}

\section{{\color{black}The time-fractional inhomogeneous Markov model}}

In this section we show how many of the above results can be partially or fully generalized to the case of a non-Markovian stochastic process that satisfies Kolmogorov's equation in the fractional derivative sense.

We hence consider the state-space $E=\{1,\dots, p\}$ and a jump process $(Z_t)$ evolving in $E$, such that its jump intensities and sojourn times are as in the Markov case, governed by $\mat{\Lambda}$. However the main difference here is that the sojourn distribution is no longer exponentially distributed, but rather Mittag-Leffler distributed, i.e. having tail function
\begin{align*}
\overline{F}_{\xi}(x)=E_\alpha(\lambda_{ii}x^{\alpha}),\quad x>0, \quad \alpha \in(0,1],
\end{align*}
where 
\begin{align}
E_{\alpha,\beta}(z)=\sum_{k=0}^\infty \frac{z^k}{\Gamma(\alpha k+\beta)},\quad z\in\mathbb{C},\quad  \beta\in\mathbb{R},\quad \alpha>0\label{mtglf_fn}
\end{align}
is the Mittag-Leffler function, and for convenience we write $E_\alpha=E_{\alpha,1}$. The Mittag-Leffler function is entire for $\beta>0$. Its matrix version can equivalently be defined in terms of an analogous series or in terms of the Cauchy integral formula.

It may be shown (cf. \cite{mml}) that, starting at zero, the transition matrix of the jump process $(Z_t)$ satisfies 
\begin{align}\label{eq:generatormmlat0}
\boldsymbol{P}(t)=E_{\alpha, 1}\left(\mat{\Lambda} t^{\alpha}\right),
\end{align}
which is the solution to the fractional differential equation  
$${ }_{0}^{c} D_{t}^{\alpha} \boldsymbol{P}(t)=\mat{\Lambda} \boldsymbol{P}(t)=\boldsymbol{P}(t) \mat{\Lambda},$$
in terms of the Caputo derivative
$${ }_{0}^{c} D_{t}^{\alpha} x(t)=\frac{1}{\Gamma(n-\alpha)} \int_{0}^{t}(t-\tau)^{n-\alpha-1} x^{(n)}(\tau) {d} \tau.$$

For the case of $\mat{\Lambda}$ having the structure \eqref{ph_structure}, i.e. with an absorbing state, then for an initial distribution $\vect{\pi}$ we get that
\begin{align*}
	\tau_Z = \inf \{ t \geq  0 \mid Z_t = p+1 \} \sim  \mbox{PH}_\alpha(\vect{\pi} , \mat{T} )
\end{align*}
is a fractional phase--type distribution with parameters $\vect{\pi}$ and $\mat{T}$. The latter distribution is precisely a phase--type distribution when $\alpha=1$, whereas for $\alpha<1$ its right tail is regularly varying with parameter $\alpha$, and hence possesses no mean. 

For a more in-depth analysis of the transition probabilities, it is necessary to regard $(Z_t)$ as a time-changed process. Thus, we introduce the $\alpha$-stable subordinator $(U_t)$ defined as a Levy process with $\E(e^{-sU_t})=\exp(-ts^\alpha)$, and its inverse $(S_t)$ defined in terms of the hitting times
\begin{align*}
S_t=\inf\{s>0\,:\,t\le U_s\},
\end{align*}
such that $\Prob(S_t\le s)=\Prob(t\le U_s)$. It follows easily that
\begin{align*}
(Z_t)\stackrel{d}{=}(Z^\circ_{S_t})
\end{align*}
where $(Z^\circ)$ is the Markov process associated to $(\vect{\pi},\mat{T})$. 
It is also straightforward to see that 
\begin{align*}
\tau_Z\stackrel{d}{=} U_{\tau_{Z^{\circ}}}\stackrel{d}{=}\tau_{Z^{\circ}}^{1/\alpha}\,U_1,
\end{align*}
as noted in \cite{mml}. For notational convenience, we will write 
\begin{align*}
X^\circ_t=(1\{Z^\circ_{t}=1\},\dots, 1\{Z^\circ_{t} = p\}), \quad \mbox{if}\quad Z^\circ_{t}\neq p+1,\quad t\ge0,
\end{align*}
and $X^\circ_t=\dagger$ otherwise.

In \cite{donatien} the following conditional transition probabilities of $(Z_t)$ were derived, generalizing \eqref{eq:generatormmlat0}.
\begin{align}
\mat{P}_{u,v}(t,x)={E}_\alpha(\mat{T}(x-u)^\alpha)1\{x\in[u,\infty)\}+\mat{I}1\{x\in[t,u)\},\end{align}
with $$[\mat{P}_{u,v}(t,x)]_{ij}=\Prob(Z_x=j|Z_t=i, U_{S_t}=u, S_t=v).$$

In \cite{donatien} the reserve of a multi-state insurance model based on $(Z_t)$ was considered. We now proceed to extend those results by considering an underlying deterministic time transform, as we did above for the Markov case. The resulting time-transformed model is of considerable interest since the absorption times may have finite mean (or higher moments) for $\alpha<1$, contrary to the homogeneous case. {\color{black} An intuitive description is that the sub-exponential behaviour (of the maximum observation behaving roughly as the sum) is cast into a more amenable environment through the inhomogeneity function.}

\subsection{A motivating example: LTC recipients}
{{\color{black}
The following example provides a starting point for the statistical analysis of fractional inhomogeneous multi-state models for life insurance modelling. We aim to motivate rather than being exhaustive, and in particular we omit time-dependent specifications, which would be analogous to the Lee-Carter model (\cite{lee1992modeling}). The latter can be straightforwardly extended from the current settings, but their statistical analysis is not straightforward, and the derivation of an EM algorithm for such extensions is out of the scope of the current paper.}

The quantification of abnormally large events and in particular of long lifetimes is a difficult task, since by definition these events are rare. In the recent article \cite{rootzen2017human}, and subsequent contributions, it is shown that individuals close to becoming super-centenarians have a mortality rate which becomes constant (exponential tail) and independent across covariates such as sex and region. Improvement in medicine and the general increasing trends on human lifetimes further suggest that this longevity effect may evolve in the future. This is, however, no obvious statistical fact, since most mortality tables gather people above $98$ in a single age category, say as $99+$. 

We consider the recent data provided in \cite{ltc_aus} for the Austrian population which are recipients of public welfare long-term-care (LTC) benefits in 2018. We focus only on the total count of people receiving well-fare benefits. We consider ages above $59$ and do not segment according to frailty level, that is, we will only study the marginal distribution of the sampled population. This data suffers from the usual binning, and collects all lifetimes above $98$ in a $99+$ category. We make the following imputation to the data based on the recent findings of exponential tail-behaviour of large human lifetimes: we remove the $99+$ counts and add a smooth continuation of the observed counts from ages $99$ to $108$ (ten years) which decrease exponentially (see Figure \ref{fig:imputed}). Even if the total counts above $98$ years are very low, we now show how an underlying fractional chain can help capture this longevity risk feature. 

\begin{figure}[!htbp]
\centering
\includegraphics[width=0.7\textwidth]{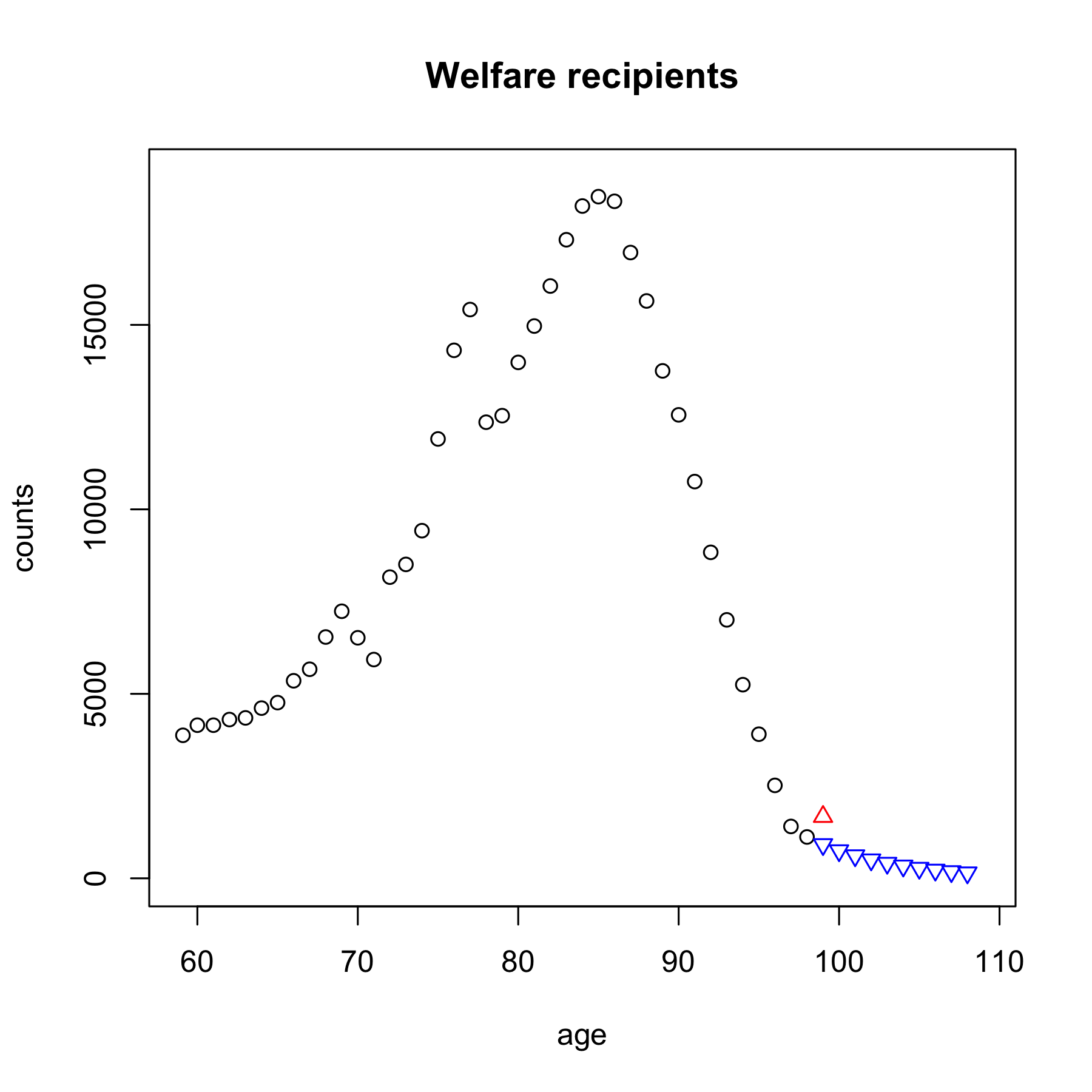}
\caption{Well-fare recipients of LTC in Austria in 2018. Circles indicate actual datapoints, the upward-pointing triangle is the total counts binned at the $99+$ category, and the downward-pointing triangles are ten years of imputed exponentially-decaying counts.} \label{fig:imputed}
\end{figure}

We fitted a three-dimensional Matrix-Gompertz distribution (as defined in Example \ref{matrix-gompertz}) and its time-fractional counterpart (anticipating Definition \ref{iph_fract} below), both having a Coxian (bi-diagonal) matrix structure. The selection of dimension and structure were taken to be as the best-fitting ones for the former distribution. The resulting density functions of the two fitted distributions are given in Figure \ref{survcurve}. Together with the reported AIC scores, it suggests that the underlying sub-exponential behaviour of the Mittag-Leffler holding times are helping with the modeling of the exponentially-decaying tail, without compromising the body of the distribution, contrary to the usual Matrix-Gompertz distribution, which does make this compromise. The rest of the fitted quantiles agree very closely between both models, since the Gompertz distribution is nonetheless a good distribution for describing the bulk of ages of human lifetimes. In the left-hand tail there is a slight miss-fit from the fractional model. This is because the stable behaviour of the Mittag-Leffler function always affects both tails of the distribution. The fractional value $\alpha=0.956$ is highly significant, as assessed by standard techniques such as a likelihood ratio test, or the variance covariance matrix derived from the inverse of the hessian (which yields a $95\%$ confidence interval of $(0.954,\, 0.958)$).

{\color{black} In general, Matrix-Gompertz distributions (and not just the simple Gompertz-Makeham law, which by now is somewhat academic) are already very accurate models for the description of human lifetimes. There is evidence that matrix models provide sturdy competition and are much more interpretable against classical lifetime models, such as the Gompertz-Makeham, or even against the Lee-Carter model when considering time dependence specifications (a deeper analysis is still required and a formal estimation procedure is still under development). Consequently, Figure 2 is remarkable since it shows the potential of fractional models in the statistical domain.}

\begin{figure}[!htbp]
\centering
\includegraphics[width=0.85\textwidth]{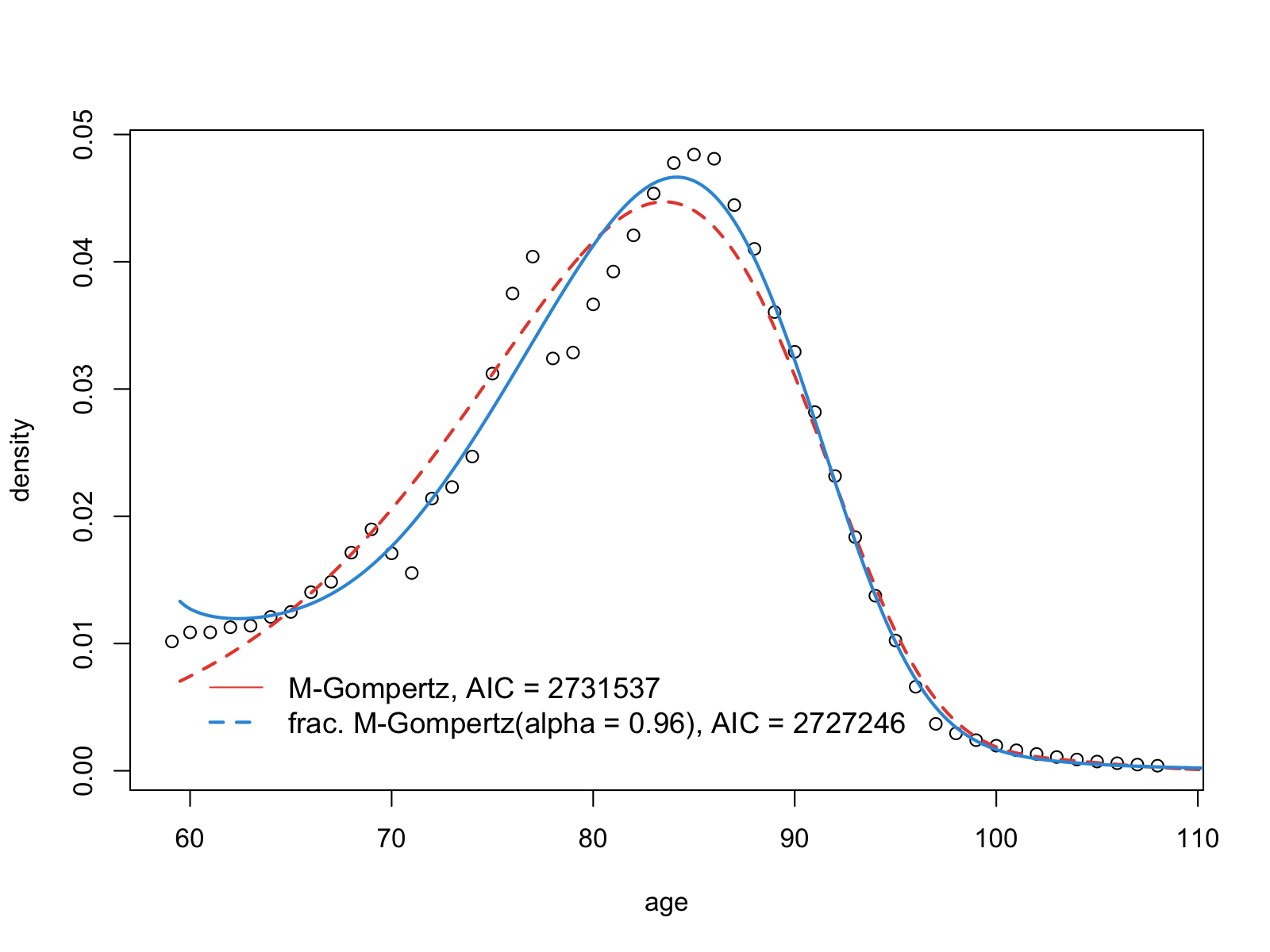}
\caption{Fitted density functions for Matrix-Gompertz distribution: time-fractional vs. standard versions.} \label{survcurve}
\end{figure}

The aim of this short analysis is not to give any sharp conclusions on the data itself, but rather to showcase how under reasonable assumptions, fractional models such as inhomogeneous matrix Mittag-Leffler (fractional phase-type) distributions can help explain certain features of data, thus making them attractive models to consider for the more theoretical calculations that follow. All statistical fitting was done using maximum-likelihood estimation, with the EM algorithm being employed for the non-fractional distribution.
}

\subsection{The inhomogeneous case}
As above, we consider a process $(X_t)$ in the state space $$E^\ast=(\vect{e}_1,\dots, \vect{e}_{p}, \dagger)$$ defined in terms of $(Z_t)$ and an inhomogeneity function $g$ by
\begin{align*}
X_t=(1\{Z_{g^{-1}(t)}=1\},\dots, 1\{Z_{g^{-1}(t)} = p\}), \quad \mbox{if}\quad Z_{g^{-1}(t)}\neq p+1,\quad t\ge0,
\end{align*}
and $X_t=\dagger$ otherwise. 

\begin{definition}\label{iph_fract}
Let $\tau=\inf\{t\ge0\,|\, X_t=\dagger\}$ be the absorption time of $X$. We say that $\tau$ is fractional inhomogeneous phase--type distributed and write $$\tau\sim \mbox{IPH}_\alpha(\vect{\pi},\mat{T},g).$$
\end{definition}

For a power transform, the absorption time of $X_t$ in terms of $\tau_Z$
\begin{align*}
	\tau =(\tau_Z)^\gamma ,\quad \gamma>0,
\end{align*}
was considered in \cite{mml} as the class of power Matrix Mittag-Leffler functions. Subsequently, \cite{mmml,mfph} extended this class to the multivariate setting. The advantage of these models from a non-life insurance point of view is that the tails are regularly varying with any possible tail index, and thus some of its members have finite moments. Presently, for life insurance, appropriate choices of $g$ will dictate the asymptotic behaviour of $\tau$ in terms of that of $\tau_Z$ in the same way that it was treated for the Markov case. 

We now present some basic properties of $\mbox{IPH}_\alpha$ distributions.

\begin{proposition}
Let $\tau\sim \mbox{IPH}_\alpha(\vect{\pi},\mat{T},g)$, then
\begin{enumerate}
\setlength{\itemsep}{5pt}
\item[(i)] $\tau\stackrel{d}{=} g(\tau_Z)$, where $\tau_Z\sim \mbox{PH}_\alpha(\vect{\pi},\mat{T})$.
\item[(ii)] $F_\tau(x)=1-\vect{\pi}E_\alpha(\mat{T}(g^{-1}(x))^{\alpha})\vect{e}.$
\item[(iii)] $f_\tau(x)=(g^{-1}(x))^{\alpha -1}\vect{\pi}E_{\alpha,\alpha}(\mat{T}(g^{-1}(x))^{\alpha})\vect{t}\lambda(x).$
\item[(iv)] $\E[h(\tau)]=\vect{\pi}\mathcal{L}^{(\alpha,\alpha)}_{h\circ g}(-\mat{T})\vect{t},$ where $$\mathcal{L}^{(\alpha,\alpha)}_{f}(u)=\int_0^\infty f(x) x^{\alpha-1}E_{\alpha,\alpha}(-ux^{\alpha})dx,$$ whenever it exists.
\end{enumerate}
\end{proposition}
\begin{proof}
The proof of {\color{black}(i) follows from the definition, (ii) and (iii) follow from (i), the definition of the Mittag-Leffler function, and the fact that $F_{\tau_Z}(x)=1-\vect{\pi}E_\alpha(\mat{T}x^{\alpha})\vect{e}$. Finally, (iv)} follows from 
\begin{align*}
\E[h(\tau)]=\int_0^\infty h(g(x)) x^{\alpha -1}\vect{\pi}E_{\alpha,\alpha}(\mat{T}x^{\alpha})\vect{t}dx=\vect{\pi}\int_0^\infty (h\circ g)(x)x^{\alpha -1}E_{\alpha,\alpha}(\mat{T}x^{\alpha})dx \,\vect{t}.
\end{align*}
\end{proof}

\subsection{{\color{black}The reserve at time $t$}}
We will now consider the conditional reserve as a means to obtain the reserve at time zero. The conditional reserve is defined conditionally in terms of the underlying subordinator and inverse subordinator and in certain situations might be of interest. However, from a practical perspective, it can be seen as an intermediate step towards efficient simulation of the reserve at time $t>0$. For $t=0$ we recover the unconditional reserve, since in that case $U_0=S_0=0$. We introduce the counting process 
\begin{align*}
N^{ij}_t=\#\{s \in(0, t], X(s-) =\vect{e}_i, X(s)=\vect{e}_j\}.
\end{align*}
The conditional reserve is defined by
\begin{align*}
Y_{t,u,v}=\E_{u,v}\left[ \int_t^n e^{-\int_t^s r}dB_s\Big |\mathcal{F}_t\right],
\end{align*}
where $\E_{t,u,v}(\cdot)=\E(\,\cdot \, | \mathcal{F}_t,U_{S_{g^{-1}(t)}}=u, S_{g^{-1}(t)}=v),$ and
$$d B_t=\sum_{i=1}^p 1\{X_{t}=\vect{e}_i\} (a_i-c_i)dt+\sum_{i\neq j} b_{ij} d N^{ij}_t$$
is the process of total benefits less premiums. For premium calculation, the equation $Y_0=0$ is easy to handle in this fractional setting, since then the conditional expectation simplifies, and thus the defining equation is given by
\begin{align*}
\E_0\left[ \int_0^ne^{-\int_0^u r}X_u (\vect{c}-\vect{a})du \right]=\sum_{i\neq j}\E_0\left[ \int_0^ne^{-\int_0^u r}1\{X_u = \vect{e}_i\}b_{ij}dN^{ij}_u \right].
\end{align*}

\begin{proposition}\label{main_prop2}
The conditional reserve at time $t$ {\color{black} with expiry date $n$ }is given by
\begin{align*}
Y_{t,u,v}=X^\circ_v \left\{\left[f^{(\alpha)}_1(\mat{\Lambda},u) + \int_t^{g(u)}e^{-\int_t^x r}dx \right](\vect{a}-\vect{c})+f^{(\alpha)}_2(\mat{\Lambda},u)(\mat{\Lambda}\cdot\mat{B})\vect{e}\right\}
\end{align*}
where
\begin{align*}
f^{(\alpha)}_1(w,u)&=\int_{g(u)}^ne^{-\int_t^x r}E_\alpha(w(g^{-1}(x)-u)^\alpha) dx,\\
f^{(\alpha)}_2(w,u)&=\int_{g(u)}^ne^{-\int_t^x r}(g^{-1}(x)-g^{-1}(t))^{\alpha-1}E_{\alpha,\alpha}(w(g^{-1}(x)-u)^\alpha)\lambda(x) dx,
\end{align*}
{\color{black}and $E_{\alpha,\alpha}$ is defined in \eqref{mtglf_fn}.}
\end{proposition}
\begin{proof}
We first write
\begin{align*}
\E_{t,u,v}\left[ \int_t^ne^{-r\int_t^x r}X_x dx\right]&= \int_t^ne^{-\int_t^x r}\E_{t,u,v}\left[X_x\right]dx.
\end{align*}
But by functional calculus, for an appropriate path $\gamma$,
\begin{align*}
&\mat{P}_{u,v}(t,x)\\
&={E}_\alpha(\mat{\Lambda}(g^{-1}(x)-u)^\alpha)1\{x\in[g(u),\infty)\} +\mat{I}1\{x\in[t,g(u)]\}\\
&=\dfrac{1}{2 \pi i} \oint_{\gamma}{E}_\alpha(w)1\{x\in[g(u),\infty)\} (w \mat{I} -\mat{\Lambda}(g^{-1}(x)-u)^\alpha )^{-1}dw+\mat{I}1\{x\in[t,g(u)]\}\\
&=\dfrac{1}{2 \pi i} \oint_{\gamma}{E}_\alpha(w(g^{-1}(x)-u)^\alpha)1\{x\in[g(u),\infty)\}(w \mat{I} -\mat{\Lambda} )^{-1}dw+\mat{I}1\{x\in[t,g(u)]\},
\end{align*}
and hence
\begin{align*}
\E_{t,u,v}&\left[ \int_t^ne^{-\int_t^x r}X_x dx \right]\\
&=\int_{g(u)}^n e^{-\int_t^x r}X^\circ_v\dfrac{1}{2 \pi i} \oint_{\gamma}E_\alpha(w(g^{-1}(x)-u)^\alpha) (w \mat{I} -\mat{\Lambda} )^{-1}dwdx\\
&\quad+\int_t^{g(u)}e^{-\int_t^x r}X^\circ_v\mat{I}dx\\
&=\dfrac{1}{2 \pi i} \oint_{\gamma}\int_{g(u)}^ne^{-\int_t^x r}X^\circ_vE_\alpha(w(g^{-1}(x)-u)^\alpha) dx(w \mat{I} -\mat{\Lambda} )^{-1}dw \\
&\quad+ X^\circ_v \int_t^{g(u)}e^{-\int_t^x r}dx\\
&=\dfrac{1}{2 \pi i} \oint_{\gamma}X^\circ_v f^{(\alpha)}_1(w) (w \mat{I} -\mat{\Lambda} )^{-1}dw+ X^\circ_v \int_t^{g(u)}e^{-\int_t^x r}dx,
\end{align*}
giving first part of the expression. We have to be a bit cautious when considering the representation of the intensity of the counting process $N_{ij}$, since the operation $A\cdot B$ between matrices does not commute with regular matrix multiplication. Thus, observe that by the tower property of conditional expectations,
\begin{align*}
&\E_{t,u,v}\left[1\{X_x = \vect{e}_i\}b_{ij}dN^{ij}_x \right]\\
&=\E_{t,u,v}\left[b_{ij}X_v^\circ e^{\mat{\Lambda}(S_{g^{-1}(x)}-S_{g^{-1}(t)})}\vect{e}_i\vect{e}_i^t\lim_{\rho\downarrow 0}\frac{e^{\mat{\Lambda}(S_{g^{-1}(x+\rho)}-S_{g^{-1}(x)})}-\mat{I}}{\rho}\vect{e}_j\right].
\end{align*}
Summing over $i$ we get 
\begin{align*}
&\sum_{i}\E_{t,u,v}\left[1\{X_x = \vect{e}_i\}b_{ij}dN^{ij}_x \right]\\
&=\E_{t,u,v}\left[X_v^\circ e^{\mat{\Lambda}(S_{g^{-1}(x)}-S_{g^{-1}(t)})}\lim_{\rho\downarrow 0}\left\{\frac{e^{\mat{\Lambda}(S_{g^{-1}(x+\rho)}-S_{g^{-1}(x)})}-\mat{I}}{\rho}\cdot \mat{B}\right\}\vect{e}_j\right]\\
&=\E_{t,u,v}\left[X_v^\circ e^{\mat{\Lambda}(S_{g^{-1}(x)}-S_{g^{-1}(t)})} \lim_{\rho\downarrow 0}\frac{S_{g^{-1}(x+\rho)}-S_{g^{-1}(x)}}{\rho} (\mat{\Lambda}\cdot \mat{B})\vect{e}_j\right]\\
&=\E_{t,u,v}\left[X_v^\circ e^{\mat{\Lambda}(S_{g^{-1}(x)}-S_{g^{-1}(t)})} \lim_{\rho\downarrow 0}\left\{\mat{\Lambda}^{-1}\frac{e^{\mat{\Lambda}(S_{g^{-1}(x+\rho)}-S_{g^{-1}(x)})}-\mat{I}}{\rho}\right\} (\mat{\Lambda}\cdot \mat{B})\vect{e}_j\right]\\
&=X_v^\circ\E_{t,u,v}\left[  \lim_{\rho\downarrow 0}\left\{\mat{\Lambda}^{-1}\frac{e^{\mat{\Lambda}(S_{g^{-1}(x+\rho)}-S_{g^{-1}(t)})}-e^{\mat{\Lambda}(S_{g^{-1}(x)}-S_{g^{-1}(t)})}}{\rho}\right\} \right](\mat{\Lambda}\cdot \mat{B})\vect{e}_j\\
&=X_v^\circ\mat{\Lambda}^{-1}\left[  \lim_{\rho\downarrow 0}\left\{\frac{E_\alpha(\mat{\Lambda} (g^{-1}(x+\rho)-u)^\alpha)-E_\alpha(\mat{\Lambda} (g^{-1}(x)-u)^\alpha)}{\rho}\right\} \right](\mat{\Lambda}\cdot \mat{B})\vect{e}_j\\
&=X_v^\circ\mat{\Lambda}^{-1}\mat{\Lambda} (g^{-1}(x)-u)^{\alpha-1}E_{\alpha,\alpha}(\mat{\Lambda}(g^{-1}(x)-u)^\alpha\lambda(x) (\mat{\Lambda}\cdot \mat{B})\vect{e}_j\\
&=X_v^\circ (g^{-1}(x)-u)^{\alpha-1}E_{\alpha,\alpha}(\mat{\Lambda}(g^{-1}(x)-u)^\alpha\lambda(x) (\mat{\Lambda}\cdot \mat{B})\vect{e}_j.
\end{align*}
Observe that we have assumed invertibility of $\mat{\Lambda}$ for clarity. The general case follows from approximation, since that term anyway gets cancelled by $\mat{\Lambda}$ arising from the derivative of the matrix Mittag-Leffler function.

Consequently,
\begin{align*}
&\sum_{i\neq j}\E_{t,u,v}\left[ \int_t^ne^{-\int_t^x r}1\{X_x = \vect{e}_i\}b_{ij}dN^{ij}_x \right]\\
&=\int_{g(u)}^n e^{-\int_t^x r}X^\circ_v(g^{-1}(x)-u)^{\alpha-1}E_{\alpha,\alpha}(\mat{\Lambda}(g^{-1}(x)-u)^\alpha\lambda(x)(\mat{\Lambda} \cdot\mat{B})\vect{e}dx\\
&=\dfrac{1}{2 \pi i}\left[ \oint_{\gamma}\int_{g(u)}^ne^{-\int_t^x r}X^\circ_v(g^{-1}(x)-g^{-1}(t))^{\alpha-1}\right.\\
&\quad\times  E_{\alpha,\alpha}(w(g^{-1}(x)-u)^\alpha\lambda(x) dx(w \mat{I} -\mat{\Lambda} )^{-1}dw (\mat{\Lambda}\cdot\mat{B})\vect{e}\Big]\\
&=\dfrac{1}{2 \pi i} \oint_{\gamma}X^\circ_v f^{(\alpha)}_2(w)du(w \mat{I} -\mat{\Lambda} )^{-1}dw (\mat{\Lambda}\cdot\mat{B})\vect{e},
\end{align*}
which yields the second part of the expression and the proof is complete.
\end{proof}

\subsection{{\color{black}The reserve at time $0$}}
We again may obtain an interpretation in terms of duality when $p+1$ is an absorbing state and $r$ is a constant. The functions involved in the calculation of the reserve at inception of the contract may be re-written as
\begin{align*}
f^{(\alpha)}_1(-w,0)&=r^{-1}\int_{0}^{g^{-1}(n)}E_\alpha(-w u^\alpha)\left\{re^{-rg(u)} \frac{dg}{du}(u)\right\}du,\\
f_2^{(\alpha)}(-w,0)&=w^{-1}\int_{0}^{n}e^{-r u}\left\{w[g^{-1}(x)]^{\alpha-1}E_{\alpha,\alpha}(-w(g^{-1}(u))^\alpha) \frac{dg^{-1}}{du}(u)\right\}du.
\end{align*}
We do not recover the Laplace transform in one case, but we may define for a function $h$ its generalized Laplace transform as
\begin{align*}
\mathcal{L}^{(\alpha)}_h(u)=\int_0^\infty E_{\alpha}(-ux^\alpha)h(x)dx,\quad \alpha\in (0,1],
\end{align*}
such that $\mathcal{L}^{(1)}_h=\mathcal{L}_h$. Then, since both 
$$E_\alpha(-rg(u)),\quad E_\alpha(-w g^{-1}(u))$$
are survival functions with corresponding densities $h_1^{(r,\alpha)}(u)$ and $h_2^{(w,\alpha)}(u)$, we obtain
\begin{align}\label{laplaces2}
f^{(\alpha)}_1(-w,0)&=r^{-1}\mathcal{L}^{(\alpha)}_{h_1^{(r,1)}}(w),\nonumber\\
f^{(\alpha)}_2(-w,0)&=w^{-1}\mathcal{L}^{(1)}_{h_2^{(w,\alpha)}}(r).
\end{align}

We now immediately obtain an analogous result to Corollary \ref{cor3}, the proof of which is virtually identical and thus omitted.
\begin{corollary}
The reserve at time $t=0$ for a contract with no expiry date is given by
\begin{align*}
Y_{0,0,0}=X^\circ_0 \left\{r^{-1}\mathcal{L}^{(\alpha)}_{h_1^{(r,1)}}(-\mat{T})(\vect{a}-\vect{c})+(-\mat{T})^{-1}\mathcal{L}^{(1)}_{h_2^{(-\mat{T},\alpha)}}(r)[(\mat{T}\cdot\mat{B})\vect{e}+\vect{t}\cdot\vect{b}]\right\}.
\end{align*}
\end{corollary}

\begin{remark}\rm
The numerical integration of the above expressions can be achieved by Jordan decomposition or diagonalization in an analogous manner to the non-fractional counterpart.
\end{remark}

\section{{\color{black}A fractional Matrix-Gompertz example}}
{\color{black} We provide an example related to a fractional Matrix-Gompertz distribution. {\color{black} The insurance contracts that we specify below are simple, and the emphasis is on the underlying fractional model. This specification paves the way for other types of models to be adapted and generalized. For instance, time dependence in the payment functions is a natural next step.}

Thus, consider the case of the inhomogeneity function $g^{-1}_\beta(x)=\beta^{-1}(\exp(\beta x)-1)$, which for $\alpha =1$ corresponds to a Markov jump process with Matrix-Gompertz absorption time, whereas in general for $\alpha<1$ we get a non-Markovian jump process with an absorption time given by $$\tau=\beta^{-1}\log(\beta \tau_Z +1),$$
where $\tau_Z\sim \mbox{PH}_\alpha(\vect{\pi},\mat{T})$. It follows that for $\alpha<1$ this random variable has exponentially decaying tails, so for $\alpha \in (0,1]$, $$\E[\tau^k]<\infty ,\quad\forall  k\ge0.$$ 
{\color{black} The latter property improves homogeneous fractional models, which are heavy-tailed and possess no moments, making them unrealistic when dealing with human lifetimes.}

Consider a Coxian underlying structure with $\beta =0.1383$ $\vect{\pi}= (1,\:0,\:0)$
and 
\begin{align*}
\mat{T}=\begin{pmatrix}
-0.1722 & 0.1585 & 0\\
0 & -0.5663 & 0.5664\\
0 & 0 & -0.0052
\end{pmatrix}.
\end{align*}

The chosen parameters are not arbitrary, since for $\alpha = 0.96$, they correspond to the fractional Matrix-Gompertz specification of Figure \ref{survcurve}.

This model corresponds to an individual which traverses four states, the fourth being an absorbing state representing death. The sojourn distribution becomes larger as the individual progresses through the states, which indicates a possible longevity risk. Moreover, for $\alpha<1$, this longevity risk is exacerbated by the sub-exponentially distributed sojourns of the underlying process $Z$.
In Figure \ref{fig:qqplot} we depict the (shifted) densities and survival curves of $\mbox{IPH}_{\alpha}(\vect{\pi},\mat{T},g_\beta)$ for $\alpha = 0.9,\:0.76,\:0.6,\:0.4$. We see that as $\alpha$ decreases, longevity grows. In fact, both tails are affected by the self-similar behaviour of the underlying stable random variable, but the left-hand side behaviour is not visible in a global scale (for this choice of parameters).


\begin{figure}[!htbp]
\centering
\includegraphics[width=1\textwidth]{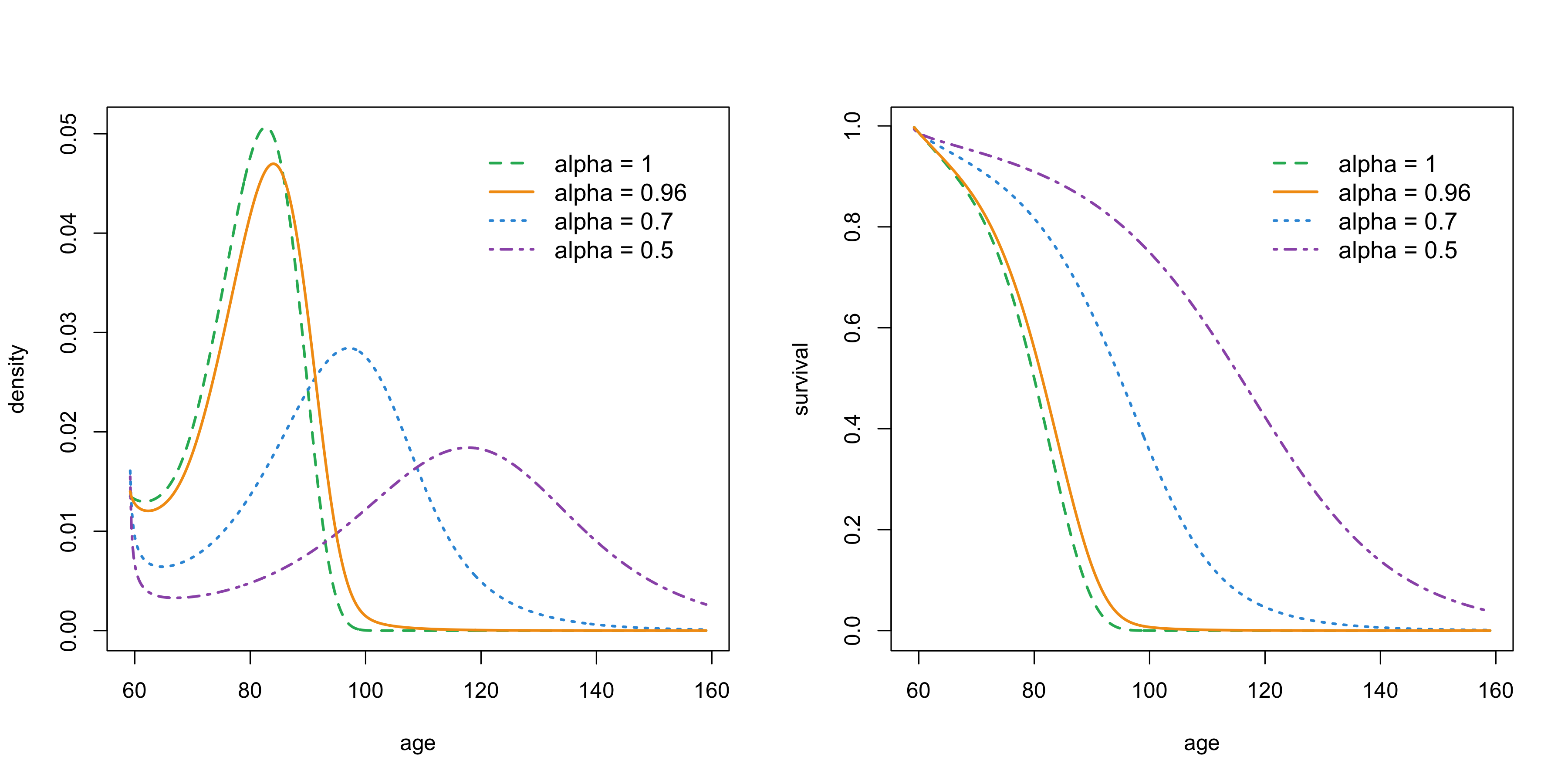}
\caption{{\color{black}Densities (left) and survival curves (right) of a fractional Gompertz IPH distribution, for varying fractional parameter $\alpha$.}} \label{fig:qqplot}
\end{figure}

We now consider two simple insurance contracts that an individual with remaining lifetime being modelled with a $\mbox{IPH}_{\alpha}(\vect{\pi},\mat{T},g_\beta)$ can purchase. More complex specifications are possible, but we presently focus on showcasing the underlying distribution and its effect on longevity.

C1. Lump sums of unit size are paid at each transition of state, including to the absorbing one, and annuities of unit rate are paid throughout the lifetime of the policyholder, or up to time $n$, whatever happens first.

C2. A lump sum of unit size is paid when transitioning from state $2$ to $3$, and a more considerable lump sum of size $50$ is paid out upon death, except if the possible contract expiry date runs out before. Annuities of rate $1/2$ are paid throughout the lifetime of the policyholder, or up to time $n$, whatever happens first.

Intuitively, C2 is seen as a larger liability for large $n$, since in that case there is a larger chance to pay out an enlarged lump sum. However, this intuition ignores longevity risk considerations, where annuities associated with C1 could become very costly in case of a longer stay in any given state.

The expected discounted liabilities (the reserve when setting premiums to zero) are shown in Figure \ref{fig:liabilities}, as a function of $n$, for different choices of $\alpha$. We observe that the longevity risk associated with smaller $\alpha$ is reflected in larger future losses for C1 when $n$ is large. When considering premium calculation through the reserve equation $Y_0=0$, it is thus crucial to correctly capture this longevity risk correctly, if it is present.

}

\begin{figure}[!htbp]
\centering
\includegraphics[width=1\textwidth]{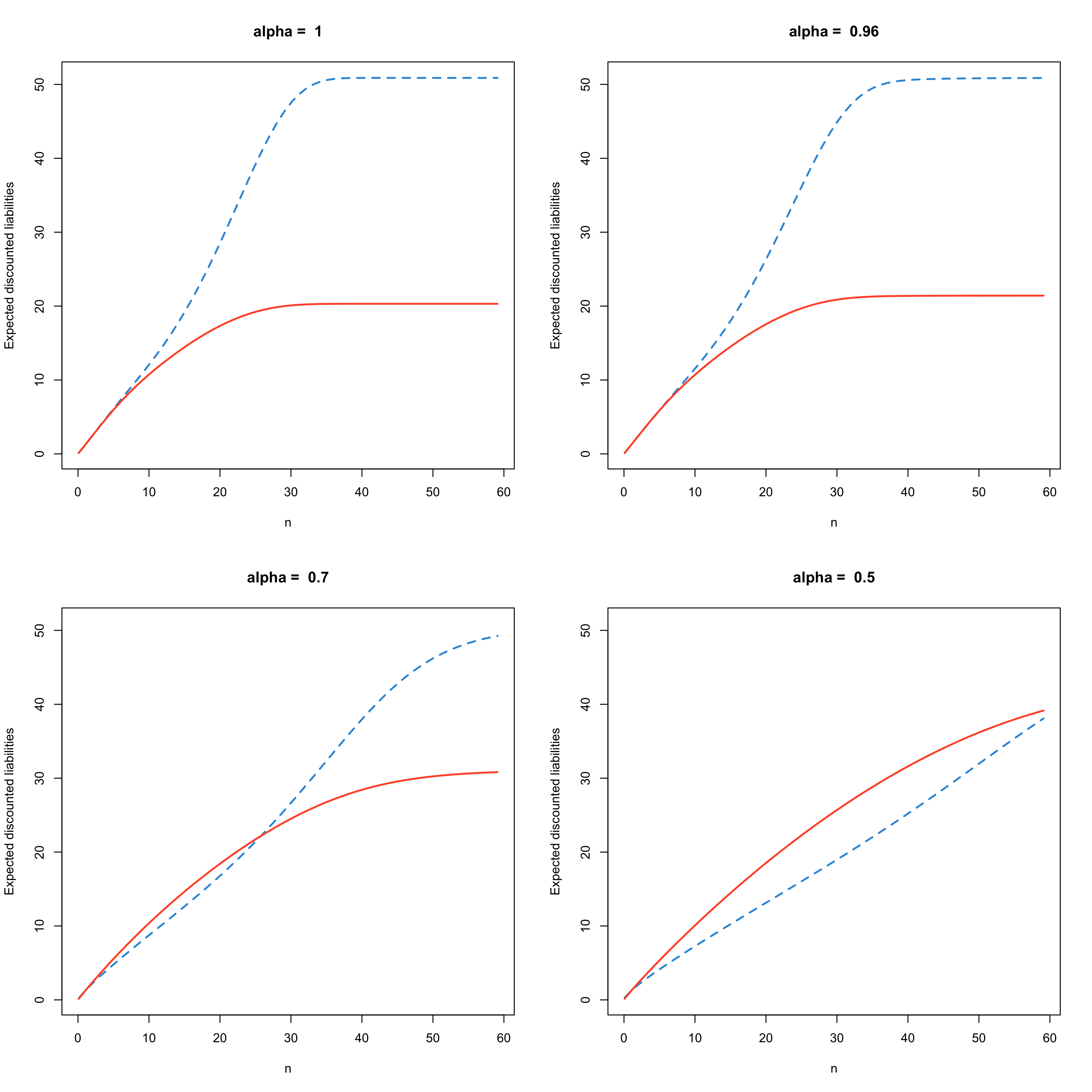}
\caption{Expected discounted liabilities corresponding to the contracts C1 (red, solid) and C2 (blue, dashed), for different choices of $\alpha$, as a function of the upper time limit $n$.} \label{fig:liabilities}
\end{figure}

\section{Conclusion}

We have shown how closed-form formulae are attainable for the reserve in multi-state life insurance models, both in the inhomogeneous Markov and inhomogeneous fractional-time non-Markov cases. The main tool enabling the generalization is the Mittag-Leffler function and its central role as the solution to the Kolmogorov fractional equation. The proposed models have very flexible absorption time distribution, and in the fractional case can account for longevity risk.

The use of product integrals and Van-Loan-type matrix decompositions in a fractional setting has not yet been developed. It is the next natural step towards painting a full picture of inhomogeneous fractional chains and their use for calculating moments of reserves associated with multi-state models with memory. {\color{black} Additionally, a full statistical account, including time-varying coefficients, is a promising line of research to bring these models one step closer to real-life applications.}

\textbf{Acknowledgement.} The author would like to acknowledge financial support from the Swiss National Science Foundation Project 200021\_191984.

\bibliographystyle{apalike}
\bibliography{ILI_final}

\end{document}